\documentclass[11pt]{amsart}
\theoremstyle{plain}
\newtheorem{Thm}{Theorem}

\newtheorem{Lem}[Thm]{Lemma}

\begin{document}

\title[Curvature flow to Nirenberg problem]
{Curvature flow to Nirenberg problem}

\author{ Minchun Hong, Li Ma}

\address{Li Ma, Department of mathematical sciences \\
Tsinghua University \\
Beijing 100084 \\
China} \email{lma@math.tsinghua.edu.cn}

\address{M.C.Hong, Department of Mathematics \\
The University of Queensland \\
Brisbane, Queensland\\
Australia} \email{hong@maths.uq.edu.au}

\dedicatory{}
\date{July. 12th, 2008}

\begin{abstract}

In this note, we study the curvature flow to Nirenberg problem on
$S^2$ with non-negative nonlinearity. This flow was introduced by
Brendle and Struwe. Our result is that the Nirenberg problems has a
solution provided the prescribed non-negative Gaussian curvature $f$
has its positive part, which possesses non-degenerate critical
points such that $\Delta_{S^2} f>0$ at the saddle points.


\end{abstract}

\thanks{$^*$ Li Ma: Corresponding
     author. The research is partially supported by the National Natural Science
Foundation of China 10631020 and SRFDP 20060003002 } \maketitle

\section{Introduction}

In the interesting paper \cite{struwe}, M.Struwe studied a heat flow
method to the Nirenberg problem on $S^2$. This kind of heat flow in
conformal geometry was considered by S.Brendle in \cite{brandle}.
Given the Riemannian metric $g$ on $S^2$ with Gaussian curvature
$K$. Using the well-known Gauss-Bonnet formula
$$
\int_{S^2} Kdv_g=4\pi,
$$
we know that $K$ has to be positive some where. This gives a
necessary condition for the Nirenberg problem on $S^2$. Assuming the
prescribed curvature function $f$ being positive on $S^2$, the heat
flow for the Nirenberg problem $S^2$ is a family of metrics of the
form $g=e^{2u(x,t)}c$ satisfying
\begin{equation}\label{flow}
 u_t=\alpha f-K, \quad x\in S^2,
\quad t>0,
\end{equation}
where $c$ is the standard spherical metric on $S^2$, $u:S^2\times
(0,T)\to R$, and $\alpha=\alpha(t)$ is defined by
\begin{equation}\label{alpha}
\alpha \int_{S^2} fdv_g=4\pi. \end{equation}
 Here $dv_g$ is the
area element with respect to the metric $g$. It is easy to see
that
$$
\alpha_t\int_{S^2}fdv_g=2\alpha \int_{S^2}(K-\alpha f)fdv_g.
$$
 M.Struwe can
show that the flow exists globally, furthermore, the flow converges
at infinity provided $f$ is positive and possesses non-degenerate
critical points such that $\Delta_{S^2} f>0$ at the saddle points.
Here $\Delta_{S^2}:=\Delta$ is the Analyst's Laplacian on the
standard 2-sphere $(S^2,c)$. Recall that $\int_{S^2}dv_c=4\pi$. The
purpose of this paper is to relax his assumption by allowing the
function $f$ to have zeros.

Since we have
$$
K=e^{2u}(-\Delta u+1),
$$
the equation (\ref{flow}) define a nonlinear parabolic equation
for $u$, and the flow exists at least locally for any initial data
$u|_{t=0}=u_0$. Clearly, we have $$
\partial_t\int_{S^2}dv_g=2\int{S^2}u_tdv_g=0.
$$
We shall assume that the initial data $u_0$ satisfies the
condition
\begin{equation}\label{positive}
\int f e^{2u}dv_c>0.
\end{equation}
We shall show that this property is preserved along the flow. It
is easy to compute that
\begin{equation}\label{curvature}
K_t=-2u_tK-\Delta_g u_t=2K(K-\alpha f)+\Delta_g(K-\alpha f),
\end{equation}
where $\Delta_g=e^{-2u}\Delta$. Using (\ref{curvature}), we can
compute the growth rate of the Calabi energy $\int_{S^2}|K-\alpha
f|^2dv_g$.

Our main result is following
\begin{Thm}\label{main}
Let $f$ be a positive somewhere, non-negative smooth function on
$S^2$ with only non-degenerate critical points on the its positive
part $f_+$. Suppose that there are at least two local positive
maxima of $f$, and at all positive valued saddle points $q$ of $f$
there holds $\Delta_{S^2}f(q)>0$. Then $f$ is the Gaussian curvature
of the conformal metric $g=e^{2u}c$ on $S^2$.
\end{Thm}

Note that this result is an extension of the famous result of
Chang-Yang \cite{chang} where only positive $f$ has been
considered. A similar result for Q-curvature flow has been
obtained in \cite{ma2}.

 For
simplifying notations, we shall use the conventions that
$dc=\frac{dv_c}{4\pi}$ and $\bar{u}=\bar{u}(t)$ defined by
$$
\int_{S^2}(u-\bar{u})dv_c=0.
$$

\section{Basic properties of the flow}

Recall the following result of Onofri-Hong \cite{ono} that
\begin{equation}\label{onof}
\int_{S^2}(|\nabla u|^2+2u)dc\geq \log(\int_{S^2} e^{2u}dc)=0,
\end{equation}
where $|\nabla u|^2$ is the norm of the gradient of the function
$u$ with respect to the standard metric $c$. Here we have used the
fact that $\int_{S^2} e^{2u}dc=1$ along the flow (\ref{flow}).

We show that this condition is preserved along the flow
(\ref{flow}). In fact, letting
$$
E(u)=\int_{S^2}(|\nabla u|^2+2u)dc
$$
be the Liouville energy of $u$ and letting
$$
E_f(u)=E(u)-\log (\int_{S^2}fe^{2u}dc)
$$
be the energy function for the flow (\ref{flow}), we then compute
that
\begin{equation}\label{monotone}
\partial_tE_f(u)=-2\int_{S^2}|\alpha f-K|^2dv_g\leq 0.
\end{equation}
One may see Lemma 2.1 in \cite{struwe} for a proof. Hence
$$
E_f(u(t))\leq E_f(u_0), \quad t>0.
$$
After using the inequality (\ref{onof}) we have
$$
\log(1/\int_{S^2}fe^{2u}dc) \leq E_f(u_0),
$$
which implies that $\int_{S^2} fe^{2u}dv_c>0$ and $$
e^{E_f(u_0)}\int_{S^2}e^{2u}dc\leq \int_{S^2}fe^{2u}dc.
$$
 Note also that
$\int_{S^2}fe^{2u}dc=1/\alpha(t)$. Hence,
$$
\alpha(t)\leq \frac{1}{e^{E_f(u_0)}}.
$$

Using the definition of $\alpha(t)$ we have
$$
\alpha(t)\geq \frac{1}{\max_{S^2}f}.
$$
We then conclude that $\alpha(t)$ is uniformly bounded along the
flow, i.e.,
\begin{equation}\label{struwe2}
\frac{1}{\max_{S^2}f}\leq\alpha(t)\leq \frac{1}{e^{E_f(u_0)}}.
\end{equation}
We shall use this inequality to replace (26) in \cite{struwe} in
the study of the normalized flow, which will be defined soon
following the work of M.Struwe \cite{struwe}. If we have a global
flow, then using (\ref{monotone}) we have
$$
2\int_0^{\infty}dt\int_{S^2} |\alpha f-K|^2dv_g\leq 4\pi
(E_f(u_0)+\log max_{S^2} f).
$$

Hence we have a suitable sequence $t_l\to\infty$ with associated
metrics $g_l=g(t_l)$ and $\alpha (t_l)\to \alpha>0$, and letting
$K_l=K(g_l)$, such that
$$
\int_{S^2}|K_l-\alpha f|^2\to 0, \; \; (t_l\to\infty).
$$
Therefore, once we have a limiting metric $g_{\infty}$ of the
sequence of the metrics $g_l$, it follows that
$K(g_{\infty})=\alpha f$. After a re-scaling, we see that $f$ is
the Gaussian curvature of the metric $\beta g_{\infty}$ for some
$\beta>0$, which implies our Theorem \ref{main}.

\section{Normalized flow and the proof of Theorem \ref{main}}

We now introduce a normalized flow. For the given flow
$g(t)=e^{2u(t)}c$ on $S^2$, there exists a family of conformal
diffeomorphisms $\phi=\phi(t):S^2\to S^2$, which depends smoothly
on the time variable $t$, such that for the metrics $h=\phi^*g$,
we have
$$
\int_{S^2} x dv_h=0, \; for \; all \; t\geq 0.
$$
Here $x=(x^1,x^2,x^3)\in S^2\subset R^3$ is a position vector of
the standard 2-sphere. Let
$$
v=u\circ \phi+\frac{1}{2}\log(det(d\phi)).
$$
Then we have $h=e^{2v}c$. Using the conformal invariance of the
Liouville energy \cite{chang}, we have
$$
E(v)=E(u),
$$
and furthermore,
$$
Vol(S^2,h)=Vol(S^2,g)=4\pi, \; for \; all \; t\geq 0.
$$

Assume $u(t)$ satisfies (\ref{flow}) and (\ref{alpha}). Then we
have the uniform energy bounds
$$
0\leq E(v)\leq E(u)=E_f(u)+\log (\int_{S^2}fe^{2u}dc)\leq
E_f(u_0)+\log(\max_{S^2}f).
$$

Using Jensen's inequality we have
$$
2\bar{v}:=\int_{S^2} 2v dc\leq \log(\int_{S^2} e^{2v}dc)=0.
$$

Using this we can obtain the uniform $H^1$ norm bounds of $v$ for
all $t\geq 0$ that
$$
\sup_t|v(t)|_{H^1(S^2)}\leq C.
$$
See the proof of Lemma 3.2 in \cite{struwe}. Using the
Aubin-Moser-Trudinger inequality \cite{Au98} we further have
$$
\sup_t\int_{S^2} e^{2pv(t)}dc\leq C(p)
$$
for any $p\geq 1$.

Note that
$$
v_t=u_t\circ \phi+\frac{1}{2} e^{-2v}div_{S^2}(\xi e^{2v})
$$
where $\xi=(d\phi)^{-1}\phi_t$ is the vector field on $S^2$
generating the flow $(\phi(t))$, $t\geq 0$, as in \cite{struwe},
formula (17), with the uniform bound
$$
|\xi|_{L^{\infty}(S^2)}\leq C\int_{S^2}|\alpha f-K|^2dv_g.
$$
With the help of this bound, we can show (see Lemma 3.3 in
\cite{struwe}) that for any $T>0$, it holds
$$
\sup_{0\leq t<T}\int_{S^2}e^{4|u(t)|}dc<+\infty.
$$
Following the method of M.Struwe \cite{struwe1} (see also Lemma
3.4 in  \cite{struwe}) and using the bound (\ref{struwe2}) and the
growth rate of $\alpha$, we can show that
$$
\int_{S^2}|\alpha f-K|^2dv_g\to 0
$$
as $t\to\infty$. Once getting this curvature decay estimate, we
can come to consider the concentration behavior of the metrics
$g(t)$. Following \cite{struwe1}, we show that

\begin{Lem} Let $(u_l)$ be a sequence of smooth functions on $S^2$
with associated metrics $g_l=e^{2u_l}c$ with $vol(S^2,g_l)=4\pi$,
$l=1,2,...$. Suppose that there is a smooth function $K_{\infty}$,
which is positive somewhere and non-negative in $S^2$ such that
$$
|K(g_l)-K_{\infty}|_{L^2(S^2,g_l)}\to 0
$$
as $l\to\infty$. Let $h_l=\phi_l^*g_l=e^{2v_l}c$ be defined as
before. Then we have either

 1) for a subsequence $l\to\infty$ we have $u_l\to u_{\infty}$ in
 $H^2(S^2, c)$, where $g_{\infty}=e^{2u_{\infty}}c$ has Gaussian
 curvature $K_{\infty}$, or

 2) there exists a subsequence, still denoted by  $(u_l)$ and a point $q\in S^2$ with $K_{\infty}(q)>0$,
  such that the metrics $g_l$ has a measure concentration that
  $$
dv_{g_l}\to 4\pi \delta_q
  $$
weakly in the sense of measures, while $h_l\to c$ in $H^2(S^2,c)$
and in particular, $K(h_l)\to 1$ in $L^2(S^2)$. Moreover, in the
latter case the conformal diffeomorphisms $\phi_l$ weakly
converges in $H^1(S^2)$ to the constant map $\phi_{\infty}=q$.
\end{Lem}

\begin{proof}
 The case 1) can be proved as Lemma 3.5 in \cite{struwe}. So we need only to
 prove the case 2). As in \cite{struwe}, we choose $q_l\in S^2$
 and radii $r_l>0$ such that
 $$
\sup_{q\in S^2}\int_{B(q,r_l)}|K(g_l)|dv_{g_l}\leq
\int_{B(q_l,r_l)}|K(g_l)|dv_{g_l}=\pi,
 $$
where $B(q,r_l)$ is the geodesic ball in $(S^2,g_l)$. Then we have
$r_l\to 0$ and we may assume that $q_l\to q$ as $l\to\infty$. For
each $l$, we introduce $\phi_l$ as in Lemma 3.5 in \cite{struwe}
so that the functions
$$
\hat{u}_l=u_l\circ \phi_l+\frac{1}{2}\log (det(d\phi_l))
$$
satisfy the conformal Gaussian curvature equation
$$
-\Delta_{R^2} \hat{u}_l=\hat{K}_l e^{2\hat{u}_l}, \; on \; R^2,
$$
where $\hat{K}_l=K(g_l)\circ\phi$ and $\Delta_{R^2}$ is the
Laplacian operator of the standard Euclidean metric $g_{R^2}$.
Note that for $\hat{g}_l=\phi^*g_l=e^{2\hat{u}_l}g_{R^2}$, we have
$$
Vol(R^2, \hat{g}_l)=Vol(S^2,g_l)=4\pi.
$$
Arguing as in \cite{struwe}, we can conclude a convergent
subsequence $\hat{u}_l\to \hat{u}_{\infty}$ in $H^2_{loc}(R^2)$
where $\hat{u}_{\infty}$ satisfies the Liouville equation
$$
-\Delta_{R^2} \hat{u}_{\infty}=\hat{K}_{\infty}(q)
e^{2\hat{u}_{\infty}}, \; on \; R^2,
$$
with the finite volume $\int_{R^2} e^{2\hat{u}_{\infty}}dz\leq
4\pi$.

We remark that blow up point only occurs at point where
$K_{\infty}>0$ provided we allow $K_{\infty}$ to change sign. We
only need to exclude the case when $K_{\infty}(q)\leq 0$. It is
clear that $K_{\infty}(q)< 0$ can not occur since there is no such a
solution on the whole plane $R^2$ \cite{ma}. If $K_{\infty}(q)=0$,
then $\hat{u}:=\hat{u}_{\infty}$ is a harmonic function in $R^2$.
Let $\bar{u}(r)$ be the average of $u$ on the circle $\partial
B_r(0)\subset R^2$. Then we have
$$
\Delta_{R^2} \bar{u}=0.
$$
Hence $\bar{u}=A+B\log r$ for some constants $A$ and $B$, where
$r=|x|$. Since $\bar{u}$ is a continuous function on $[0,\infty)$,
we have $\bar{u}=A$, which is impossible since we have by Jensen's
inequality that
$$
2\pi\int_0^{\infty}e^{2\bar{u}(r)} rdr\leq \int_{R^2}
e^{2\hat{u}_{\infty}}dz\leq 4\pi.
$$

 We now have $K_{\infty}(q)> 0$. Recall that we have assumed
 $K_{\infty}\geq 0$.
 So we can follow the proof of Lemma 3.5 in
\cite{struwe}. In fact, using the classification result of Chen-Li
\cite{chenli} we know that $\hat{u}_{\infty}$ can be obtained from
stereographic projection $S^2\to R^2$ with
$$
\int_{R^2} K_{\infty}(q)e^{2\hat{u}_{\infty}}dz= 4\pi.
$$

Thus for any large $R>0$, we have error $\circ(1)\to 0$ as $l\to
\infty$ such that
$$
4\pi=\int\int_{R^2} K_{\infty}(q)e^{2\hat{u}_{\infty}}dz\leq
\int_{B_R(q)} K_l dv_l+\circ(1)\leq \int_{B_R(q)} |K_l|
dv_l+\circ(1).
$$

Then the remaining part can be derived as in \cite{struwe}. We
confer to \cite{struwe} for the full proof.

\end{proof}

With this understanding, we can do the same finite-dimensional
dynamics analysis as in section 4 in \cite{struwe}. Then arguing
as in section 5 in \cite{struwe} we can prove Theorem \ref{main}.
By now the argument is well-known, we omit the detail and refer to
\cite{struwe} for full discussion.

\end{document}